\documentclass[12pt]{article}
\usepackage{latexsym, amssymb, amsmath, amsthm, amscd}
\usepackage[all,cmtip]{xy}
\usepackage{amsmath}
\usepackage{amsthm}
\usepackage{amssymb}
\usepackage{amsfonts}
\usepackage{amsrefs}
\usepackage{color,authblk}
\usepackage{tikz}

\usepackage{amscd} 
\usepackage{comment}
\usepackage{mathrsfs}
\usepackage[all]{xy}
\usepackage{graphicx}
\usepackage{authblk}
\usepackage{hyperref}
\usepackage{fullpage}
\usepackage[all,cmtip]{xy}
\usepackage{relsize}
\usepackage{amsmath,amscd}
\usepackage{amssymb}
\usepackage{stackrel}
\usepackage{tikz-cd}
\usepackage{tikz}
\usetikzlibrary{arrows}
\usetikzlibrary{matrix}
\usepackage[mathscr]{euscript}

\usepackage[all]{xy}

\usepackage{hyperref}

\numberwithin{equation}{section} \DeclareMathSizes{2}{10}{12}{13}
\makeatletter
\newcommand*{\doublerightarrow}[2]{\mathrel{
  \settowidth{\@tempdima}{$\scriptstyle#1$}
  \settowidth{\@tempdimb}{$\scriptstyle#2$}
  \ifdim\@tempdimb>\@tempdima \@tempdima=\@tempdimb\fi
  \mathop{\vcenter{
    \offinterlineskip\ialign{\hbox to\dimexpr\@tempdima+1em{##}\cr
    \rightarrowfill\cr\noalign{\kern.5ex}
    \rightarrowfill\cr}}}\limits^{\!#1}_{\!#2}}}

\newcommand{\leftrarrows}{\mathrel{\raise.75ex\hbox{\oalign{%
  $\scriptstyle\leftarrow$\cr
  \vrule width0pt height.5ex$\hfil\scriptstyle\relbar$\cr}}}}
\newcommand{\lrightarrows}{\mathrel{\raise.75ex\hbox{\oalign{%
  $\scriptstyle\relbar$\hfil\cr
  $\scriptstyle\vrule width0pt height.5ex\smash\rightarrow$\cr}}}}
\newcommand{\Rrelbar}{\mathrel{\raise.75ex\hbox{\oalign{%
  $\scriptstyle\relbar$\cr
  \vrule width0pt height.5ex$\scriptstyle\relbar$}}}}

\makeatletter
\def\leftrightarrowsfill@{\arrowfill@\leftrarrows\Rrelbar\lrightarrows}
\newcommand{\xleftrightarrows}[2][]{\ext@arrow 3399\leftrightarrowsfill@{#1}{#2}}
\makeatother

\newtheorem{thm}{Proposition}[section]
\newtheorem{Thm}[thm]{Theorem}

\newtheorem{lem}[thm]{Lemma}
\newtheorem{defn}[thm]{Definition}

\title{A note on measurings and higher order Hochschild homology of algebras}
\author{Abhishek Banerjee\footnote{Department of Mathematics, Indian Institute of Science, Bangalore, India. Email: abhishekbanerjee1313@gmail.com} $\qquad\qquad$ Surjeet Kour\footnote{Department of Mathematics, Indian Institute of Technology, Delhi, India. Email: koursurjeet@gmail.com} }
\date{}
\begin{document}

	\maketitle 
	
	\begin{abstract}
		We know that coalgebra measurings behave like generalized maps between algebras. In this note, we show that coalgebra measurings between commutative algebras induce morphisms between higher order Hochschild homology groups of algebras. By higher order Hochschild homology, we mean the the Hochschild homology groups of a commutative algebra with respect to a simplicial set as introduced by Pirashvili.
	\end{abstract}
	
	\medskip
	{MSC(2020) Subject Classification: 16T15}

	\medskip
	{Keywords: coalgebra measurings, higher order Hochschild homology}  
	
	\medskip
	\section{Introduction}

Let $K$ be a field and let $Vect_K$ be the category of vector spaces over $K$. Let $A$, $B$ be $K$-algebras. We recall (see \cite{Sweed}) that a coalgebra measuring from $A$ to $B$ is a $K$-linear map $\psi:C
	\longrightarrow Hom_k(A,B)$,  where $C$ is a $K$-coalgebra and such that
	\begin{equation}\label{2.1eq}
	\psi(s)(aa')=\sum \psi(s_{(1)})(a)\psi(s_{(2)})(a') \qquad \qquad \psi(s)(1_A)=\epsilon_C(s)\cdot 1_B
	\end{equation} for any $a$, $a'\in A$. Here, $\epsilon_C:C\longrightarrow K$ is the counit on $C$ and  the coproduct $\Delta_C$ on the coalgebra $C$ is denoted by $\Delta(s)=\sum s_{(1)}\otimes s_{(2)}$ in Sweedler notation for each $s\in C$. We will often drop the summation and write $\Delta_C(s)=s_{(1)}\otimes s_{(2)}$ for $s\in C$.  When there is no danger of confusion, we will write $\psi(s)(a)$ simply as $s(a)$ for any $s\in C$ and $a\in A$. Accordingly, we have $s(aa')=s_{(1)}(a)\cdot s_{(2)}(a')$ for any $s\in C$ and $a$, $a'\in A$. 

\smallskip
The study of coalgebra measurings provides an enrichment of the category of algebras over coalgebras. This enrichment is often called the ``Sweedler hom'' between algebras (see for instance, \cite{AJ}) and is related to the Sweedler dual (see \cite{Porst}). For more on measurings and their adaptations in various categorical frameworks, we refer the reader for instance to  \cite{Bat0}, \cite{bim1}, \cite{bim2}, \cite{Vas1}, \cite{V1}, \cite{Vas2}, \cite{Vas3}. 

\smallskip
Since coalgebra measurings behave like generalized maps between algebras, can measurings be used to induce maps between (co)homology theories of algebras? We have considered this question in  \cite{BK0}, where we used measurings to construct maps between Hochschild homology groups of algebras. In \cite{BK1}, we have shown that measurings lead to induced maps between (co)homology theories on Hopf algebroids. In \cite{TP}, Pirashvili introduced Hochschild homology groups of commutative algebras with respect to simplicial sets, which he called higher order Hochschild homology groups. In this brief note, we show that coalgebra measurings between commutative algebras induce morphisms on higher order Hochschild homology groups. 

	\section{Measurings and morphisms on higher order Hochschild homology}

	\smallskip
	When $C$ is a cocommutative coalgebra,  $\psi:C
	\longrightarrow Hom_k(A,B)$ is said to be a cocommutative measuring from $A$ to $B$.  From now onward, $A$ and $B$ will always denote commutative $K$-algebras and $C$ will always denote a cocommutative $K$-coalgebra. Let $M$ be an $A$-module and $N$ be a $B$-module. A $C$-comodule measuring (see \cite{Bat1}, \cite{V1}) from $M$ to $N$  is a $K$-linear map   $\phi:D\longrightarrow Hom_k(M,N)$,  where $D$ is a $C$-comodule and   such that
\begin{equation}\label{2.2eq}
\phi(t)(am)=\sum \psi(t_{(0)})(a)\cdot \phi(t_{(1)})(m)
\end{equation}	for $t\in D$, $a\in A$ and $m\in M$. Here, structure map of $D$ as a $C$-comodule is expressed as $\delta_D(t)=\sum t_{(0)}\otimes t_{(1)}\in C\otimes D$ for 
$t\in D$. As with coalgebra measurings between algebras, we will often write the expression in \eqref{2.2eq} simply as $t(am)=t_{(0)}(a)\cdot t_{(1)}(m)$ for $t\in D$, 
$a\in A$ and $m\in M$. 

\smallskip
In \cite{TP}, Pirashvili introduced the notion of higher order Hochschild homology of a commutative algebra $A$ with coefficients in a module $M$. Let $\Gamma$ be the category of finite sets with basepoint. We denote by $\Gamma-Mod$ the category of left $\Gamma$-modules, i.e., functors $\Gamma\longrightarrow Vect_K$. Let $Sets_\star$ be the category of pointed sets. Then if $R:\Gamma\longrightarrow Vect_K$ is a left $\Gamma$-module, $R$ can be extended to a functor that we denote by 
\begin{equation}\label{2.3ext}
\widetilde R:Sets_\star \longrightarrow Vect_K\qquad T\mapsto \underset{\Gamma\ni S\rightarrow T}{colim}\textrm{ }R(S)
\end{equation} Let $\Delta$ denote the usual simplicial category, whose objects are ordered sets $\{0,1,...,k\}$ for $k\geq 0$, and whose morphisms are order preserving maps 
(see, for instance, \cite[$\S$ 6]{Loday}). By definition, a pointed simplicial set $Y$ is a functor $Y:\Delta^{op}\longrightarrow Sets_\star$. Accordingly, given a pointed simplicial set 
$Y$ and a left $\Gamma$-module $R$, we have the composition
\begin{equation}\label{2.4e}
\Delta^{op} \xrightarrow{\qquad Y\qquad} Sets_\star \xrightarrow{\qquad \widetilde R\qquad}Vect_K
\end{equation} which gives a simplicial vector space $\widetilde R(Y_\bullet)$. 

\smallskip For $k\geq 0$, we denote by $[k]\in \Gamma$ the set $[k]:=\{0,1,...,k\}$ with basepoint $0$. If $A$ is a commutative $K$-algebra and $M$ is a $K$-module, we have a left $\Gamma$-module $\mathcal L(A,M):\Gamma\longrightarrow Vect_K$ that is determined by the associations (see Loday \cite{Lod0}):
\begin{equation}\label{e2.5}
[k]\mapsto M\otimes A^{\otimes k} 
\end{equation} For $k$, $l\geq 0$ and a  morphism $f:[k]\longrightarrow [l]$ in $\Gamma$, the induced morphism is determined as follows
\begin{equation}\label{32.6}
\begin{array}{c}
\mathcal L(A,M)(f): M\otimes A^{\otimes k} \longrightarrow  M\otimes A^{\otimes l} \qquad\qquad 
m\otimes a_1\otimes ...\otimes a_k\mapsto n\otimes b_1\otimes ...\otimes b_l\\
b_j:=\underset{i\in f^{-1}(j)}{\prod} a_i\qquad n:=m\cdot \underset{i\in f^{-1}(0)}{\prod} a_i \\
\end{array}
\end{equation} for each $1\leq j\leq l$. In particular, when $M=A$, we set $\mathcal L(A):=\mathcal L(A,A)$.

\begin{defn}\label{D2.1} (see \cite{TP}) Let $A$ be a commutative $K$-algebra, $M$ be an $A$-module and $Y$ be a pointed simplicial set. Then, the Hochschild homology groups
$HH_\bullet^Y(A,M)$ of order $Y$ for the algebra $A$ with coefficients in $M$ are given by the homologies
$\{HH_n^Y(A,M):=H_n( \widetilde{\mathcal L(A,M)}(Y_\bullet))\}_{n\geq 0}$  of the simplicial vector space obtained by the composition
\begin{equation}\label{2.7ee}
{\mathcal L(A,M)}(Y_\bullet): \Delta^{op} \xrightarrow{\qquad Y\qquad} Sets_\star \xrightarrow{\qquad \widetilde{\mathcal L(A,M)}\qquad}Vect_K
\end{equation} If $g:Y\longrightarrow Z$ is a map of pointed simplicial sets, the induced map on Hochschild homology groups is denoted by $HH^g_\bullet(A,M): HH_\bullet^Y(A,M)\longrightarrow HH_\bullet^Z(A,M)$. When $M=A$,  the Hochschild homology groups
 of order $Y$ for the algebra $A$ are denoted by $HH_\bullet^Y(A):=HH_\bullet^Y(A,A)$. 
\end{defn}  

\begin{lem}\label{L2.2} Let $\phi:D\longrightarrow Hom_k(M,N)$ be a $C$-comodule measuring from the $A$-module $M$ to the $B$-module $N$. Then for each $t\in D$, we have an induced map
\begin{equation}\label{2.9tg}
\mathcal L^\phi(t):\mathcal L(A,M)\longrightarrow \mathcal L(B,N)
\end{equation} of $\Gamma$-modules. 

\end{lem}

\begin{proof} Let $t\in D$. For each $k\geq 0$, we define 
\begin{equation}\label{2.10tg}
\begin{array}{c}
\mathcal L^\phi(t)[k]:\mathcal L(A,M)[k]\longrightarrow \mathcal L(B,N)[k]\\
 m\otimes  a_1\otimes ...\otimes a_k\mapsto  \phi(t_{(0)})(m)\otimes \psi(t_{(1)})(a_1)\otimes ... \otimes \psi(t_{(k)})(a_k)= t_{(0)}(m)\otimes t_{(1)}(a_1)\otimes ... \otimes t_{(k)}(a_k)\\
\end{array}
\end{equation}
We now consider $k$, $l\geq 0$ and a morphism $f:[k]\longrightarrow [l]$ in $\Gamma$. Let $m\otimes a_1\otimes ...\otimes a_k\in  M\otimes A^{\otimes k}$.  Using the definitions in \eqref{2.2eq}, \eqref{32.6} and the fact that $C$ is cocommutative, we note that
\begin{equation}\label{nat2}
\begin{array}{l}
(\mathcal L(B,N)(f)\circ \mathcal L^\phi(t)[k])(m\otimes a_1\otimes ...\otimes a_k)\\ \qquad \qquad =\mathcal L(B,N)(f)( t_{(0)}(m)\otimes t_{(1)}(a_1)\otimes ... \otimes t_{(k)}(a_k))\\
\qquad \qquad= \left(t_{(0)}(m)\cdot \underset{i_0\in f^{-1}(0)}{\prod} t_{(i_0)}(a_{i_0}) \right)\otimes \left(\underset{i_1\in f^{-1}(1)}{\prod} t_{(i_1)}(a_{i_1})\right)\otimes ... \otimes \left(\underset{i_l\in f^{-1}(l)}{\prod} t_{(i_l)}(a_{i_l})\right)\\
\qquad \qquad = t_{(0)}\left(m\cdot \underset{i_0\in f^{-1}(0)}{\prod} a_{i_0}\right)\otimes t_{(1)}\left(\underset{i_1\in f^{-1}(1)}{\prod} a_{i_1}\right)\otimes ...\otimes t_{(l)}\left(\underset{i_l\in f^{-1}(0)}{\prod} a_{i_l}\right)\\
\qquad \qquad = \mathcal L^\phi(t)[l]\left(\left(m\cdot \underset{i_0\in f^{-1}(0)}{\prod} a_{i_0}\right)\otimes 
\left(\underset{i_1\in f^{-1}(1)}{\prod} a_{i_1}\right)\otimes ...\otimes \left(\underset{i_l\in f^{-1}(0)}{\prod} a_{i_l}\right) \right)\\
\qquad \qquad =   (\mathcal L^\phi(t)[l]\circ \mathcal L(A,M)(f))(m\otimes a_1\otimes ...\otimes a_k)\\
\end{array}
\end{equation}
From \eqref{nat2}, it is clear that the maps in \eqref{2.10tg} determine a natural transformation $\mathcal L^\phi(t):\mathcal L(A,M)\longrightarrow \mathcal L(B,N)$ of
 functors from $\Gamma$ to $Vect_K$.
\end{proof}

\begin{Thm}
\label{P2.3} Let $A$, $B$ be commutative $K$-algebras, and let $\psi:C\longrightarrow Hom_K(A,B)$ be a cocommutative measuring from $A$ to $B$.  Let $\phi:D\longrightarrow Hom_k(M,N)$ be a $C$-comodule measuring from the $A$-module $M$ to the $B$-module $N$. Let $Y$ be a pointed simplicial set. Then for each $t\in D$, there is an induced morphism 
\begin{equation}\label{2.11dc}
\phi^Y_\bullet(t):HH_\bullet^Y(A,M)\longrightarrow HH_\bullet^Y(B,N)  
\end{equation} on Hochschild homology groups of order $Y$.  Moreover, if $g:Y\longrightarrow Z$ is a map of pointed simplicial sets, we have
$HH^g_\bullet(B,N)\circ \phi^Y_\bullet(t)=\phi^Z_\bullet(t)\circ HH^g_\bullet(A,M)$.
\end{Thm}

\begin{proof}
By Lemma \ref{L2.2}, the $C$-comodule measuring $\phi:D\longrightarrow Hom_k(M,N)$ induces a morphism $\mathcal L^\phi(t):\mathcal L(A,M)\longrightarrow \mathcal L(B,N)$ of $\Gamma$-modules for each $t\in D$, and hence a morphism of functors $\widetilde{\mathcal L^\phi(t)}:\widetilde{\mathcal L(A,M)}\longrightarrow \widetilde{\mathcal L(B,N)}$ from $Sets_\ast$ to $Vect_K$. By the definition in \eqref{2.7ee}, this induces a morphism 
$\widetilde{\mathcal L^\phi(t)}(Y_\bullet):\widetilde{\mathcal L(A,M)}(Y_\bullet)\longrightarrow \widetilde{\mathcal L(B,N)}(Y_\bullet)$ of simplicial vector spaces and hence the morphisms $\phi^Y_\bullet(t):HH_\bullet^Y(A,M)\longrightarrow HH_\bullet^Y(B,N) $. Since $\mathcal L^\phi(t)$ is a morphism of $\Gamma$-modules, it is also clear that the morphism in \eqref{2.11dc} is well behaved with respect to maps of pointed simplicial sets. 
\end{proof}

When $M=A$, $N=B$, $D=C$ and the measuring $\phi=\psi:C\longrightarrow Hom_K(A,B)$, for each $s\in C$, the maps in \eqref{2.11dc} reduce to morphisms
\begin{equation}\label{2.12dc}
\psi^Y(s):HH_\bullet^Y(A)\longrightarrow HH^Y_\bullet(B)
\end{equation}
of Hochschild homology groups of order $Y$.


\begin{bibdiv}
		\begin{biblist}



\bib{AJ}{article}{
   author={Anel, M.},
      author={Joyal, A.},
   title={Sweedler theory for (co)algebras and the bar-cobar constructions},
   journal={arXiv:1309.6952},
   date={2013},
}
		
		\bib{BK0}{article}{
   author={Banerjee, A.},
   author={Kour, S.},
   title={On measurings of algebras over operads and homology theories},
   journal={Algebr. Geom. Topol.},
   volume={22},
   date={2022},
   number={3},
   pages={1113--1158},
}
			
			\bib{BK1}{article}{
   author={Banerjee, A.},
   author={Kour, S.},
   title={Measurings of Hopf algebroids and morphisms in cyclic (co)homology
   theories},
   journal={Adv. Math.},
   volume={442},
   date={2024},
   pages={Paper No. 109581, 48},
}

\bib{Bat0}{article}{
   author={Batchelor, M.},
   title={Difference operators, measuring coalgebras, and quantum group-like
   objects},
   journal={Adv. Math.},
   volume={105},
   date={1994},
   number={2},
   pages={190--218},
}

\bib{Bat1}{article}{
   author={Batchelor, M.},
   title={Measuring comodules---their applications},
   journal={J. Geom. Phys.},
   volume={36},
   date={2000},
   number={3-4},
   pages={251--269},
}

	\bib{bim1}{article}{
   author={Grunenfelder, L.},
   author={Mastnak, M.},
   title={On bimeasurings},
   journal={J. Pure Appl. Algebra},
   volume={204},
   date={2006},
   number={2},
   pages={258--269},
}

\bib{bim2}{article}{
   author={Grunenfelder, L.},
   author={Mastnak, M.},
   title={On bimeasurings. II},
   journal={J. Pure Appl. Algebra},
   volume={209},
   date={2007},
   number={3},
   pages={823--832},
}



\bib{Vas1}{article}{
   author={Hyland, M.},
   author={L\'{o}pez Franco, I.},
   author={Vasilakopoulou, C.},
   title={Hopf measuring comonoids and enrichment},
   journal={Proc. Lond. Math. Soc. (3)},
   volume={115},
   date={2017},
   number={5},
   pages={1118--1148},
}


\bib{V1}{article}{
   author={Hyland, M.},
   author={L\'{o}pez Franco, I.}
   author={Vasilakopoulou, C.}
   title={Measuring comodules and enrichment},
   journal={ arXiv 1703.10137},
   date={2017},
}

\bib{Lod0}{article}{
   author={Loday, J.~-L.},
   title={Op\'erations sur l'homologie cyclique des alg\`ebres commutatives},
   journal={Invent. Math.},
   volume={96},
   date={1989},
   number={1},
   pages={205--230},
}

\bib{Loday}{book}{
   author={Loday, J.~-.L},
   title={Cyclic homology},
   series={Grundlehren der mathematischen Wissenschaften},
   volume={301},
   edition={2},
   note={Appendix E by Mar\'ia O. Ronco;
   Chapter 13 by the author in collaboration with Teimuraz Pirashvili},
   publisher={Springer-Verlag, Berlin},
   date={1998},
   pages={xx+513},
}

\bib{TP}{article}{
   author={Pirashvili, T.},
   title={Hodge decomposition for higher order Hochschild homology},
   journal={Ann. Sci. \'Ecole Norm. Sup. (4)},
   volume={33},
   date={2000},
   number={2},
   pages={151--179},
}

	\bib{Porst}{article}{
   author={Porst, H.-E.},
   author={Street, R.},
   title={Generalizations of the Sweedler dual},
   journal={Appl. Categ. Structures},
   volume={24},
   date={2016},
   number={5},
   pages={619--647},
}


\bib{Sweed}{book}{
   author={Sweedler, M.~E.},
   title={Hopf algebras},
   series={Mathematics Lecture Note Series},
   publisher={W. A. Benjamin, Inc., New York},
   date={1969},
   pages={vii+336},
}
			


\bib{Vas2}{article}{
   author={Vasilakopoulou, C.},
   title={On enriched fibrations},
   journal={Cah. Topol. G\'{e}om. Diff\'{e}r. Cat\'{e}g.},
   volume={59},
   date={2018},
   number={4},
   pages={354--387},
}

\bib{Vas3}{article}{
   author={Vasilakopoulou, C.},
   title={Enriched duality in double categories: $\mathcal{V}$-categories and
   $\mathcal{V}$-cocategories},
   journal={J. Pure Appl. Algebra},
   volume={223},
   date={2019},
   number={7},
   pages={2889--2947},
}
			
		\end{biblist}
		
	\end{bibdiv}
\end{document}